\documentclass[12pt]{article}
\usepackage[pdftex]{color, graphicx}
\usepackage{setspace}
\usepackage[utf8]{inputenc}
\usepackage{amsmath}
\usepackage[pdftex,bookmarks,colorlinks]{hyperref}
\usepackage{latexsym}
\usepackage{amsfonts}
\usepackage{amssymb}
\usepackage{amsthm}
\usepackage{paralist}
\usepackage{mathrsfs}
\usepackage{enumitem}
\usepackage[T1]{fontenc}
\begin{document}\newtheorem{thm}{Theorem}
\newtheorem{cor}[thm]{Corollary}
\newtheorem{lem}{Lemma}
\theoremstyle{remark}\newtheorem{rem}{Remark}
\theoremstyle{definition}\newtheorem{defn}{Definition}
\title{Banach space valued $H^p$ spaces with $A_p$ weight}
\author{Sakin Demir\\
Agri Ibrahim Cecen University\\ 
Faculty of Education\\
Department of Basic Education\\
04100 A\u{g}r{\i}, Turkey\\
e-mail: sakin.demir@gmail.com}
\date{12.09.2022}
\maketitle
\renewcommand{\thefootnote}{}



\renewcommand{\thefootnote}{\arabic{footnote}}
\setcounter{footnote}{0}

\maketitle
\begin{abstract}In this research we introduce the Banach space valued $H^p$ spaces with $A_p$ weight, and prove the following results:\\
\indent
Let $\mathbb{A}$ and $\mathbb{B}$ Banach spaces, and $T$ be a convolution operator mapping $\mathbb{A}$-valued functions into $\mathbb{B}$-valued functions, i.e.,
$$Tf(x)=\int_{\mathbb{R}^n}K(x-y)\cdot f(y)\, dy,$$
where $K$ is a strongly measurable function defined on $\mathbb{R}^n$ such that $\|K(x)\|_{\mathbb{B}}$ is locally integrable away from the origin. Suppose that $w$ is a positive weight function defined on $\mathbb{R}^n$, and that 
\begin{enumerate}[label=\upshape(\roman*), leftmargin=*, widest=iii]
\item For some $q\in [1, \infty ]$, there exists a positive constant $C_1$ such that
$$\int_{\mathbb{R}^n}\|Tf(x)\|^q_{\mathbb{B}}w(x)\, dx\leq C_1\int_{\mathbb{R}^n}\|f(x)\|_{\mathbb{A}}^q w(x)\,dx$$
for all $f\in L^q_{\mathbb{A}}(\mathbb{R}^n)$.
\item There exists a positive constant $C_2$ independent of $y\in\mathbb{R}^n$ such that
$$\int_{|x|>2|y|}\|K(x-y)-K(x)\|_{\mathbb{B}}\, dx<C_2.$$
\end{enumerate}
Then there exists a positive constant $C_3$ such that
$$\|Tf\|_{L^1_{\mathbb{B}}(w)}\leq C_3\|f\|_{H^1_{\mathbb{A}}(w)}$$
 for all $f\in H^1_{\mathbb{A}}(w)$.\\
\indent
 Let $w\in A_1$. Assume that $K\in L_{\rm{loc}}(\mathbb{R}^n\backslash \{0\})$ satisfies
$$\|K\ast f\|_{L^2_{\mathbb{B}}(w)}\leq C_1\|f\|_{L^2_{\mathbb{A}}(w)}$$
and

$$\int_{|x|\geq C_2|y|}\|K(x-y)-K(x)\|_{\mathbb{B}}w(x+h)\, dx\leq C_3w(y+h)\;\;\;(\forall y\neq 0, \forall h\in\mathbb{R}^n)
$$
for certain absolute constants $C_1$, $C_2$, and $C_3$. Then there exists a positive constant $C$ independent of $f$ such that
$$\|K\ast f\|_{L^1_{\mathbb{B}}(w)}\leq C\|f\|_{H^1_{\mathbb{A}}(w)}$$
for all $f\in H^1_{\mathbb{A}}(w)$.
\end{abstract}
\section{Introduction}
\indent

Extending $L^p$ spaces to Banach space valued $L^p$ spaces first started with the work of A. Benedek {\it{et al}}~\cite{bcp}. J. Bourgain~\cite{jb} extended some part of their results to a lattice with UMD-property. Later, the results of A. Benedek {\it{et al}}~\cite{bcp} have been reconstructed with a little more modern notations  by J. L. Rubio de Francia {\it{et al}}~\cite{jlrdffjrjlt}.\\
\indent
Obviously, Banach space valued setting is more general than the usual structure because we have a Banach space norm instead of absolute value. When a theorem can be extended from $L^p$ spaces to Banach space valued $L^p$ spaces it becomes a much more powerful theorem than its initial version. \\
\indent
Let us first recall some basic definitions and theorems from Banach space valued $L^p$ theory:\\
Let $\mathbb{B}$ be a Banach space, and $p<\infty$. By $L_{\mathbb{B}}^p=L_{\mathbb{B}}^p(\mathbb{R}^n)$ we denote the Bochner-Lebesgue space consisting of all $\mathbb{B}$-valued (strongly) measurable functions $f$ in $\mathbb{R}^n$ such that
$$\|f\|_{L_{\mathbb{B}}^p}=\left(\int_{\mathbb{R}^n}\|f(x)\|_{\mathbb{B}}^p\, dx\right)^{1/p}<\infty .$$
For $p=\infty$, norm of an element of $L^\infty (\mathbb{B})=L_{\mathbb{B}}^{\infty}(\mathbb{R}^n)$ is
$$\|f\|_{L_{\mathbb{B}}^\infty}=\textrm{ess}\sup \|f(x)\|_{\mathbb{B}}<\infty$$
and $L_{\rm{c}}^{\infty}(\mathbb{B})$ denotes the space of all compactly supported members of $L^\infty (\mathbb{B}).$\\
Let $f$ be a locally integrable $\mathbb{B}$-valued function, and $1\leq r\leq\infty$. We define the maximal functions
$$M_rf(x)=\sup_{x\in Q}\left(\frac{1}{|Q|}\int_Q\|f(y)\|^r_{\mathbb{B}}\, dy\right)^{1/r}$$
and
$$f^{\sharp}(x)=\sup_{x\in Q}\frac{1}{|Q|}\int_Q\|f(y)-f_Q\|_{\mathbb{B}}\, dy,$$
where $Q$ denotes an arbitrary cube in $\mathbb{R}^n$ and $f_Q$ is the average of $f$ over $Q$, an element of $\mathbb{B}$.\\
Given a weight $w$ on $\mathbb{R}^n$, we denote by $L_{\mathbb{B}}^p(w)$ the space of all functions satisfying
$$\|f\|^p_{L^p_{\mathbb{B}}(w)}=\int_{\mathbb{R}^n}\|f(x)\|^p_{\mathbb{B}}w(x)\, dx<\infty .$$
When $p=\infty$, $L^{\infty}_{\mathbb{B}}(w)$ will be taken to mean $L^\infty_{\mathbb{B}}(\mathbb{R}^n)$ and 
$$\|f\|_{L^\infty_{\mathbb{B}}(w)}=\|f\|_{L^\infty_{\mathbb{B}}}.$$

\section{Banach Space Valued $H^p$ Spaces with $A_p$ Weight}
\indent

Analogous to the classical weighted Hardy spaces we can also define the weighted Hardy spaces $H^p_{\mathbb{B}}(w)$ of $\mathbb{B}$-valued functions for $p>0$. Let $\phi$ be a function in $\mathscr{S}(\mathbb{R}^n)$, the Schwartz space of rapidly decreasing smooth functions, satisfying $\int_{\mathbb{R}^n}\phi (x)\, dx=1$. Define
$$\phi_t(x)=t^{-n}\phi (x/t), \;\;\;\; t>0,\;\;\;\; x\in \mathbb{R}^n,$$
and the maximal function $f^\ast$ by
$$f^\ast (x)=\sup_{t>0}\|f\ast \phi_t (x)\|_{\mathbb{B}}.$$
Then $H^p_{\mathbb{B}}(w)$ consists of those tempered distributions $f\in \mathscr{S}^\prime (\mathbb{R}^n)$ for which $f^{\ast}\in L^p_{\mathbb{B}}(w)$ with $\|f\|_{H^p_{\mathbb{B}}(w)}=\|f^\ast\|_{L^p_{\mathbb{B}}(w)}.$\\
As in the classical case these spaces can also be characterized in terms of atoms in the following way.
\begin{defn}Let $0<p\leq 1\leq q\leq\infty $ and $p\neq q$ such that $w\in A_q$ with critical index $q_w$. Set $[\,\cdot\, ]$ the integer function. For $s\in\mathbb{Z}$ satisfying $s\geq [n(q_w/p-1)]$, a $\mathbb{B}$-valued function $a$ defined on $\mathbb{R}^n$ is called a $(p,q,s)$-atom with respect to $w$ if
\begin{enumerate}[label=\upshape(\roman*), leftmargin=*, widest=iii]
\item $a\in L^q_{\mathbb{B}}(w)$ and is supported on a cube $Q$,
\item $\|a\|_{ L^q_{\mathbb{B}}(w)}\leq w(Q)^{1/q-1/p}$,
\item $\int_{\mathbb{R}^n}a(x)x^{\alpha}\, dx=0$ for every multi-index $\alpha$ with $|\alpha|\leq s$.
\end{enumerate}
The $\mathbb{B}$-valued atom defined above is called $(p,q,s)$-atom centered at $x_0$ with respect to $w$ (or $w-(p,q,s)$-atom centered at $x_0$), where $x_0$ is the center of the cube $Q$.
\end{defn}
\begin{lem}Let $a$ be any $\mathbb{B}$-valued $w-(p,q,s)$-atom supported in a cube $Q$. Then we have
$$\int_Q\|a(x)\|^p_{\mathbb{B}}w(x)\, dx\leq 1.$$
\end{lem}
\begin{proof}
Let $a$ be any $\mathbb{B}$-valued $w-(p,q,s)$-atom. It is clear that $a\in L^p_{\mathbb{B}}(w)$ and $\|a\|_{ L^p_{\mathbb{B}}(w)}\leq 1$, since by H\"older's inequality 
\begin{align*}
\int_Q\|a(x)\|^p_{\mathbb{B}}w(x)\, dx&\leq \|a^p\|_{L^r_{\mathbb{B}}(w)}\left(\int_Qw(x)\, dx\right)^{1/r^\prime}\\
&= \|a\|^p_{L^q_{\mathbb{B}}(w)}\cdot w(Q)^{1-p/q}\\
&\leq 1,
\end{align*}
where $r=q/p$ and $1/{r^\prime}=1-1/r=1-p/q$.
\end{proof}
Analog to the classical case $H^p_{\mathbb{B}}(w)$ can be characterized by $\mathbb{B}$-valued $w-(p,q,s)$-atoms.\\
We state the following few theorems without proof since their proofs are similar to the scalar case, i.e., one only needs to replace the absolute value with the $\mathbb{B}$-norm in the proofs for classical cases.
\begin{thm}\label{acharhpw}Let $w\in A_\infty$ and $0<p\leq 1$. For each $f\in H^p_{\mathbb{B}}(w)$, there exists a sequence of $\mathbb{B}$-valued $(p, \infty, N)$-atoms with respect to $w$ and a sequence $\{\lambda_i\}$ of real numbers with $\sum_j |\lambda_i|^p\leq C\|f\|^p_{H^p_{\mathbb{B}}(w)}$ such that 
$$f(x)=\sum_j\lambda_ja_j(x);\;\;\; \;\;\;(\lambda_j\in \ell^p)$$
both in the sense of distribution and in the $H^p_{\mathbb{B}}(w)$ norm.
\end{thm}
\indent
Let $H^{p,q,s}_{\mathbb{B}}(w)$ denote the space consisting of tempered distributions admitting a decomposition 
$$f(x)=\sum_j\lambda_ja_j(x);\;\;\; \;\;\;(\lambda_j\in \ell^p),$$
where $a_i$'s are $\mathbb{B}$-valued $w-(p,q,s)$-atoms and $\sum_i |\lambda_i|^p<\infty$. For fixed functions $w$ and $f\in H^p_{\mathbb{B}}(w)$, we also set
$$\mathscr{N}_{p,q,s}(f)=\inf_{\{\lambda_i\}}\left\{\left(\sum_i|\lambda_i|^p\right)^{1/p}:f=\sum_i\lambda_ia_i \;\text{is an atomic decomposition}\right\}.$$
\begin{thm} If both triples $(p,q,N)$ and $(p,q_2,N)$ satisfy the conditions in definition of $\mathbb{B}$-valued $w$-atom, then
$$H^{p,q_q,N}_{\mathbb{B}}(w)=H^{p,q_2,N}_{\mathbb{B}}(w)$$
and, for all $q$, the gauges $\mathscr{N}_{p,q,N}(f)$ are equivalent.
\end{thm}
\begin{thm}For $0<p\leq 1\leq q\leq\infty$ and $p\neq q$, every $\mathbb{B}$-valued  $(p,q,N)$-atom with respect to $w$ is in $H^p_{\mathbb{B}}(w)$, and its $H^p_{\mathbb{B}}(w)$-norm is bounded by a constant independent of the atom.
\end{thm}
\begin{thm}All spaces $H^{p,q,s}_{\mathbb{B}}(w)$ coincide with $H^p_{\mathbb{B}}(w)$ and $\mathscr{N}_{p,q,s}(f)\approx \|f\|_{H^p_{\mathbb{B}}(w)}$ provided that the triple $(p,q,s)$ satisfies the conditions in the definition of $\mathbb{B}$-valued $w$-atom.
\end{thm}
\begin{defn}Let $\mathbb{B}$  be a Banach space. For $0<p\leq 1\leq q\leq\infty$ and $p\neq q$, let $w\in A_q$ with critical index $q_w$ and critical index $r_w$ for the reverse H\"older condition. Set $s\geq N$, $\epsilon >\max\{sr_w(r_w-1)^{-1}n^{-1}+(r_w-1)^{-1}, 1/p-1\}$, $a=1-1/p+\epsilon$, and $b=1-1/p+\epsilon$. A $\mathbb{B}$-valued $(p,q,s,\epsilon)$-molecule centered at $x_0$ with respect to $w$ (or $w-(p,q,s,\epsilon )$-molecule centered at $x_0$) is a function $M\in L^q_{\mathbb{B}}(w)$ satisfying
\begin{enumerate}[label=\upshape(\roman*), leftmargin=*, widest=iii]
\item $M(x)\cdot w(I^{x_0}_{|x-x_0|})^b\in L^q_{\mathbb{B}}(w)$,
\item $\|M\|^{a/b}_{L^q_{\mathbb{B}}(w)}\cdot \|M(x)\cdot w(I^{x_0}_{|x-x_0|})^b\|^{1-a/b}_{L^q_{\mathbb{B}}(w)}\equiv \mathscr{R}_w(M)<\infty$,
\item $\int_{\mathbb{R}^n}x^{\alpha}\, dx=0$ for every multi-index $\alpha$ with $|\alpha|\leq s$.
\end{enumerate}
\end{defn}
In the above definition $\mathscr{R}_w(M)$ is called the molecular norm of $M$ with respect to $w$ (or $w$-molecular norm of $M$). If $w(x)\equiv$ constant, then $q_w=1$ and $r_w=\infty$.\\

\section{The Results}
Let $\mathbb{A}$ and $\mathbb{B}$ be Banach spaces, and $T$ be a convolution operator mapping $\mathbb{A}$-valued functions into $\mathbb{B}$-valued functions, i.e.,
$$Tf(x)=\int_{\mathbb{R}^n}K(x-y)\cdot f(y)\, dy,$$
where $K$ is a strongly measurable function defined on $\mathbb{R}^n$ such that $\|K(x)\|_{\mathbb{B}}$ is locally integrable away from the origin.\\
\indent
The following theorem is our first result:
\begin{thm}\label{wh1thm}Let $\mathbb{A}$ and $\mathbb{B}$ be Banach spaces, and $T$ be a convolution operator mapping $\mathbb{A}$-valued functions into $\mathbb{B}$-valued functions. Suppose that $w$ is a positive weight function defined on $\mathbb{R}^n$, and that 
\begin{enumerate}[label=\upshape(\roman*), leftmargin=*, widest=iii]
\item\label{t5i} For some $q\in [1, \infty ]$, there exists a positive constant $C_1$ such that
$$\int_{\mathbb{R}^n}\|Tf(x)\|^q_{\mathbb{B}}w(x)\, dx\leq C_1\int_{\mathbb{R}^n}\|f(x)\|_{\mathbb{A}}^q w(x)\,dx$$
for all $f\in L^q_{\mathbb{A}}(\mathbb{R}^n)$.
\item\label{t5ii} There exists a positive constant $C_2$ independent of $y\in\mathbb{R}^n$ such that
$$\int_{|x|>2|y|}\|K(x-y)-K(x)\|_{\mathbb{B}}\, dx\leq C_2.$$
\end{enumerate}
Then there exists a positive constant $C_3$ such that
$$\|Tf\|_{L^1_{\mathbb{B}}(w)}\leq C_3\|f\|_{H^1_{\mathbb{A}}(w)}$$
 for all $f\in H^1_{\mathbb{A}}(w)$
\end{thm}
\begin{proof}Given a ball $U=U(x_0; R)$ in $\mathbb{R}^n$ with center $x_0$ and radius $R$, and denoting by $\widetilde{U}$ the double ball, $\widetilde{U}=U(x_0;2R)$, we first claim that

$$\int_{\mathbb{R}^n-\widetilde{U}}\|Tf(x)\|_{\mathbb{B}}w(x)\, dx\leq C\|f\|_{L^1_{\mathbb{A}}(w)}$$
for every $f\in L_{\mathbb{A}}^1(w)$ supported in $U$ such that $\int f(x)\, dx=0$. But, for such a function $f$,
$$Tf(x)=\int_U\{K(x-y)-K(x-x_0)\}\cdot f(y)\, dy\;\;\; \;\; (x\in \widetilde{U} )$$
and therefore
\begin{align*}
\int_{\mathbb{R}^n-\widetilde{U}}\|Tf(x)\|_{\mathbb{B}}w(x)\, dx\hspace{9cm} \\
\leq \int\int_{|x-x_0|\geq 2R>2|y-x_0|}\|\{K(x-y)-K(x-x_0)\}\cdot f(y)\|_{\mathbb{B}}\, dy w(x)\, dx\\
\leq C\int _{|y-x_0|<R}\|f(y)\|_{\mathbb{A}}w(y)\, dy,\hspace{6.7cm}
\end{align*}
which proves our claim.\\
Let now $a$ be an $\mathbb{A}$-valued atom with supporting cube $Q$, and let $U$ be the smallest ball containing $Q$, and $\widetilde{Q}$ as before. Then there exits  a constant $C_1>0$ such that
$$\int_{\mathbb{R}^n-\widetilde{U}}\|Ta(x)\|_{\mathbb{B}}w(x)\, dx\leq C_1.$$
On the other hand, since
$$\int_{\mathbb{R}^n}\|Ta(x)\|^q_{\mathbb{B}}w(x)\, dx\leq C_2\int_{\mathbb{R}^n}\|a(x)\|_{\mathbb{A}}^qw(x)\,dx,$$
we have
\begin{align*}
\int_{\widetilde{U}}\|Ta(x)\|_{\mathbb{B}}w(x)\, dx& \leq C_3 \|a(x)\|_{L^q_A(w)}(C_nw(Q))^{1/q^{\prime}}\\
&\leq \text{Constant.}
\end{align*}
\end{proof}
Our second result is the following:
\begin{thm}\label{h1l1bw} Let $w\in A_1$. Assume that $K\in L_{\rm{loc}}(\mathbb{R}^n\backslash \{0\})$ satisfies
$$\|K\ast f\|_{L^2_{\mathbb{B}}(w)}\leq C_1\|f\|_{L^2_{\mathbb{A}}(w)}$$
and

$$\int_{|x|\geq C_2|y|}\|K(x-y)-K(x)\|_{\mathbb{B}}w(x+h)\, dx\leq C_3w(y+h)\;\;\;(\forall y\neq 0, \forall h\in\mathbb{R}^n)
$$
for certain absolute constants $C_1$, $C_2$, and $C_3$. Then there exists a constant $C$ independent of $f$ such that
$$\|K\ast f\|_{L^1_{\mathbb{B}}(w)}\leq C\|f\|_{H^1_{\mathbb{A}}(w)}$$
for all $f\in H^1_{\mathbb{A}}(w)$.
\end{thm}
\begin{proof} Because of the atomic decomposition of a function in $H^1_{\mathbb{B}}(w)$, it suffices to show that
$$\|K\ast a\|_{L^1_{\mathbb{B}}(w)}\leq C$$
for any $\mathbb{B}$-valued $w-(1,2,0)$-atom $a$ with constant $C$ independent of the choice of $a$. Let us first consider a weighted $1$-atom $a$ centered at $0$ with support $\textrm{supp}(f)\subset I_R$, we have 
$$\|a\|_{L^2_{\mathbb{B}}(w)}\leq w(I_R)^{-1/2}$$
and
$$\int_{I_R}a(x)\, dx=0.$$
Thus, we have
\begin{align*}
\int_{|x|\geq C_2\sqrt{n}R}\|K\ast a(x)\|_{\mathbb{B}}w(x)\, dx \hspace{7cm}\\
=\int_{|x|\geq C_2\sqrt{n}R}\left\|\int_{I_R}\{K(x-y)-K(x)\}a(y)\, dy\right\|_{\mathbb{B}}w(x)\, dx\hspace{1cm}\\
\leq \int_{I_R}\|a(y)\|_{\mathbb{B}}\, dy\int_{|x|\geq C_2|y|}\|K(x-y)-K(x)\|_{\mathbb{B}}w(x)\, dx\hspace{1.1cm}\\
\leq C_3\int_{I_R}\|a(y)\|_{\mathbb{B}}w(y)\, dy\hspace{6.5cm}\\
\leq C_3.\hspace{9.8cm}
\end{align*}
Furthermore, we have by Schwarz's inequality and the doubling condition,
\begin{align*}
\int_{|x|<C_2\sqrt{n}R}\|K\ast a(x)\|_{\mathbb{B}}w(x)\, dx&\leq \|K\ast a\|_{L^2_{\mathbb{B}}(w)}\left(\int_{|x|<C_2\sqrt{n}R}w(x)\, dx\right)^{1/2}\\
&\leq C_1\|a\|_{L^2_{\mathbb{B}}(w)}w(C_2\sqrt{n}I_R)^{1/2}\\
&\leq C.
\end{align*}
So in both cases for any $\mathbb{B}$-valued $w-(1,2,0)$-atom $a$ centered at the origin we have obtained
$$\|K\ast a\|_{L^1_{\mathbb{B}}(w)}\leq C.$$
\indent
Let now $a$ be a $\mathbb{B}$-valued $w-(1,2,0)$-atom centered t $x_0\in\mathbb{R}^n$. Then $b(x)=a(x-x_0)$ is a $\mathbb{B}$-valued $w_1-(1,2,0)$-atom centered  at 0, where $w_1(x)=w(x-x_0)\in A_1$. Furthermore, $K$ satisfies
$$\|K\ast b\|_{L^2_{\mathbb{B}}(w_1)}\leq C_1\|b\|_{L^2_{\mathbb{B}}(w_1)}$$
and
$$\int_{|x|\geq C_2|y|}\|K(x-y)-K(x)\|_{\mathbb{B}}w_1(x)\leq C_3w_1(y)\;\;\;\; (\forall y\neq 0).$$
Thus, we have as above
$$\|K\ast b\|_{L^1_{\mathbb{B}}(w_1)}\leq C.$$
Hence, we obtain
$$\|K\ast a\|_ {L^1_{\mathbb{B}}(w)}= \|K\ast b\|_{L^1_{\mathbb{B}}(w_1)}\leq C$$
as desired.
\end{proof}
\section{An Application}
Let $f$ be a measurable functions defined on $\mathbb{R}$, and for each $n\in\mathbb{Z}$ define the averaging operator
$$A_nf(x)=\frac{1}{2^n}\int_{x}^{x+2^n}f(y)\, dy.$$
Consider the variation operator
$$\mathcal{V}f(x)=\left(\sum_{-\infty}^\infty |A_nf(x)-A_{n-1}f(x)|^s\right)^{1/s}$$
for $2\leq s <\infty$.\\
Given a locally integrable function $f$ we define the sequence-valued operator $T$ as follows:
\begin{align*}
Tf(x)&=\{A_nf(x)-A_{n-1}f(x)\}_n\\
&=\left\{\int_{\mathbb{R}}\left(\frac{1}{2^n}\chi_{(-2^n,0)}(x-y)-\frac{1}{2^{n-1}}\chi_{(-2^{n-1},0)}(x-y)\right)f(y)\, dy\right\}_n\\
&=\int_{\mathbb{R}}K(x-y)\cdot f(y)\, dy,
\end{align*}
where $K$ is the sequence-valued function
$$K(x)=\left\{\frac{1}{2^n}\chi_{(-2^n,0)}(x)-\frac{1}{2^{n-1}}\chi_{(-2^{n-1},0)}(x)\right\}_{n\in\mathbb{Z}}=\left\{K_n(x)\right\}_{n\in\mathbb{Z}}.$$
It is clear that 
$$\|Tf(x)\|_{\ell^s(\mathbb{Z})}=\mathcal{V}f(x).$$
It is proven in Lemma 1 of S. Demir~\cite{sdemir3} that $K$ satisfies the $D_r$ condition for $r\geq 1$. For $r=1$ this condition is equivalent to Theorem~\ref{wh1thm} \ref{t5ii} known as H\"ormander condition with $\mathbb{B}=\ell^s(\mathbb{Z})$ for $s\geq 2$.\\
Also, Theorem 2 of S. Demir~\cite{sdemir3} shows that Theorem~\ref{wh1thm} \ref{t5i} is satisfied for $1\leq q<\infty$ with the absolute value as $\mathbb{A}$. This shows  that  Theorem~\ref{wh1thm} can be applied to $Tf$, and thus for $s\geq 2$ there exists a positive constant $C$ such that
$$\|Tf\|_{L^1_{\ell^s(\mathbb{Z})}(w)}=\|\mathcal{V}f\|_{L^1(w)}\leq C\|f\|_{H^1(w)}$$
for all $f\in H^1(w)$.


\end{document}